\documentclass[11pt, letter]{article}
\usepackage{amsmath}
\usepackage{amssymb}
\usepackage{amsfonts}
\usepackage{mathrsfs}
\usepackage{amsthm}
\usepackage{epsfig,epic}
\usepackage[english]{babel}
\usepackage{graphics}
\usepackage{graphicx}
\usepackage{eucal}
\usepackage{amscd}
\usepackage{amssymb}
\usepackage{fancyhdr}
\usepackage[hang]{subfigure}
\usepackage{array}
\usepackage[nottoc]{tocbibind}
\usepackage{verbatim}

\newtheorem{lemma}{Lemma}[section]
\newtheorem{prop}{Proposition}[section]

\newtheorem{teo}{Theorem}[section]

\newtheorem{rem}{Remark}[section]
\newenvironment{proofof}[2]{\begin{proof}[Proof of #1 \ref{#2}.]}{\end{proof}}
\newcommand{\be}{\begin{equation}}
\newcommand{\ee}{\end{equation}}
\newcommand{\bes}{\begin{equation*}}
\newcommand{\ees}{\end{equation*}}

\newcommand{\ud}{\mathrm{d}}
\newcommand{\BS}[2]{S_{#2}\left(#1\right)}

\newcommand{\G}{\mathscr{G}}
\newcommand{\h}[1]{\hat{#1}}
\newcommand{\C}[1]{\mathcal{C}{([#1]) }}
\newcommand{\hC}[1]{\h{\mathcal{C}}{([#1]) }}
\newcommand{\Ci}{\mathcal{C}}
\newcommand{\st}{\, : \, }
\newcommand{\Cn}{\mathscr{C}_n}

\newcommand{\Cp}[1]{\mathcal{C}{([#1]) }}

\newcommand{\Cnp}{\mathscr{C}^+_n}

\begin{document}
%\ETDS{1}{3}{XX}{2007}

\title{Renewal-type Limit Theorem for the Gauss Map and Continued Fractions\footnote{To appear in \emph{Ergodic Theory and Dynamical Systems}.}}

\author{YAKOV G.~SINAI and CORINNA ULCIGRAI}
\date{$18$ February $2007$}

\maketitle

\begin{center}{\it Mathematics Department, Princeton University, Fine Hall, Washington road, 08544-1000 Princeton, NJ, USA\\
(sinai@math.princeton,edu, ulcigrai@math.princeton.edu) }
\end{center}

\begin{abstract} 
In this paper we prove the following renewal-type limit theorem. Given $\alpha \in (0,1)\backslash \mathbb{Q}$ and $R>0$,  let $q_{n_R}$ be the first denominator of the convergents of $\alpha$ which exceeds $R$.
% Let $q_n$ be denominators of the continued fraction expansion of an irrational $\alpha \in (0,1)$ and given $T>0$, let $n_R$ be the first $ n \in \mathbb{N}$ such that $ q_n > R$. 
The main result in the paper is that the ratio $q_{n_R}/R$ has a limiting distribution as $R$ tends to infinity. The existence of the limiting distribution uses mixing of a special flow over the natural extension of the Gauss map.  
\end{abstract}

\begin{flushright}
{\it Dedicated to the memory of Bill Parry. }
\end{flushright}
\vspace{6mm}
\section{Introduction.}
\subsection{Main Result.}
%\subsection{Main Result.}\label{defns}
For $\alpha\in (0,1)\backslash \mathbb{Q}$, denote the continued fraction expansion of $\alpha$ by
\bes
\alpha = \frac{1}{a_1 + \frac{1}{a_2 + \frac{1}{\dots}}} = [a_1, a_2, \dots , a_n, \dots ]
\ees
where $a_n \in \mathbb{N}_+$ are the entries of the continued fraction and $\{ p_n/q_n \}_{n\in \mathbb{N}_+}$ are the convergents of $\alpha$, i.e. $ p_n/q_n  =  [a_1, a_2, \dots , a_n ]$ with $(p_n,q_n)=1$ . 
%Let $\mu_1$ denote the Gauss measure on $(0,1)$, i.e. the probability measure given by the density $\rho(\alpha)= (\log 2 (1+\alpha))^{-1}$.

We prove the following theorem. %, which is of interest in various applications.
\begin{teo}\label{main}
Given $R>0$, introduce
\bes
n_R = \min \{ n \in \mathbb{N} \, | \, \,  q_n > R \} .
\ees
Fix also $N\geq0$. Then the ratio $\frac{q_{n_R}}{R}$ and the entries $a_{n_R-k}$ for $0\leq k < N$ 
%,  $a_{n_R-1}, \dots, a_{n_R-N}$
 have a joint limiting distribution, as $R$ tends to infinity, with respect to the Gauss measure $\mu_{1}$ given by the density $\frac{\ud \mu_1 }{\ud \alpha}=(\ln 2 (1+\alpha))^{-1}$.
%Consider the ratio $\frac{q_n}{R}$. The question of estimating whether $\frac{q_n}{T}$ has a limiting distribution 
\end{teo}
Theorem \ref{main} means that for each $N\geq0$ there exists a probability measure $P_N$ on $(1,\infty) \times \mathbb{N}_+^N$ such that for all $a, b >1$, $c_k\in \mathbb{N}_+$, $0\leq k < N$,
\be \label{limiting}
\begin{split}
 \mu_1  \{ \alpha  \st a < \frac{q_{n_R}}{R} < b,  \quad a_{n_R-k} = c_k, \, 0\leq k & < N \}  \xrightarrow{R\rightarrow \infty} 
\\& P_N \left( 
(a,b) \times \{ c_0 \} \times \dots \times \{ c_{N-1} \} \right). 
\end{split}
\ee
In  Theorem \ref{main}, instead of $\mu_1$, one can consider any absolutely continuous measure, but we do not dwell on this.

\subsection{Applications.}
Theorem \ref{main} is useful in many applications. As an example, we refer to the two following papers. %, in which proofs of the existence of limiting distributions in different contexts are  based to the existence of the limiting distribution in Theorem \ref{main}. 
In \cite{BS:lim}, the authors consider the problem of limiting behavior of larges  Frobenius numbers, initially investigated by V.I.~Arnold \cite{Ar:wea}. If $a=(a_1,\dots, a_n)$ is an $n$-tuple of positive integers which are coprime, the Frobenius number $F(a)$ is the smallest $F$ such that any integer $t\geq F$ can be written in the form $t=\sum_{j=1}^n x_j a_j$ where $x_j$ are non-negative integers. Let $\Omega_N$ be the ensemble of all coprime $n$-tuples with entries less than $N$ with the uniform probability distribution.
% In \cite{BS:lim} Bourgain and the first author prove tightness of the distributions of  $\frac{1}{N^{3/2}}F(a)$ on  $\Omega_N$ as $N$ tends to infinity. 
In the case $n=3$, the existence of the limiting distribution of $\frac{1}{N^{3/2}}F(a)$ as $N$ tends to infinity is proved  using a discrete version of Theorem \ref{main}.

In \cite{SU:lim}, the following trigonometric sums are considered:
\bes
\frac{1}{N}\sum_{n=0}^N  \frac{1}{1- e^{2\pi i (n\alpha + x)} },\qquad (x,\alpha)\in (0,1)\times(0,1),
\ees
where $(0,1)\times(0,1)$ is endowed with the uniform probability distribution. The authors prove that such trigonometric sums (and, more generally, the Birkhoff sums of a function with a singularity of type $1/x$ over a rotation) have a non-trivial joint limiting distribution in $x$ and $\alpha$ as $N$ tends to infinity. 
%Theorem \ref{main} is used as an important igredient of the proof. 
Also in this case the  proof of the existence of the limiting distribution is based on the existence of the limiting distribution in Theorem \ref{main}.

\subsection{Outline.}
The main idea of the proof of Theorem \ref{main} is to reformulate the problem in terms of a certain special flow over the natural extension of the Gauss map and to exploit mixing of the flow to prove the existence of the limiting distribution. The definitions of the natural extension of the Gauss map and of special flow are recalled in \S \ref{definitionsection}. The reduction to a special flow is shown in \S\ref{reductionsec} and the existence of the limiting distribution is proved in \S\ref{limitingsec}.  The same special flow was considered also in \cite{DS:sta} and the proof that the special flow is mixing is recalled in \S \ref{mixingsection}.
\section{Definitions.}\label{definitionsection}
\subsection{Gauss map.}
Let $\G$ be the Gauss map, i.e. the transformation on $(0,1)$ given by $\alpha\mapsto \G(\alpha)= \{ \frac{1}{\alpha} \}$, where $\{\cdot \} $ denotes the fractional part. The Gauss measure $\mu_1$
% = \rho(\alpha) \ud \alpha$
 is invariant under $\G$.
 The sequence $\{a_n\}_{n\in \mathbb{N}_+}$ can be seen as a symboling coding for the orbit $\{ \G ^n \alpha \}_{n\in \mathbb{N}}$, since $a_n = [(\G ^{n-1} (\alpha))^{-1}]$ where $[\cdot]$ denotes the integer part. 
%We will often identify $\alpha$ with the sequence $\{a_n\}_{n\in \mathbb{N}}\in \mathbb{N}^{\mathbb{N}}$. 
%Points in $(0,1)\backslash \mathbb{Q}$ are given by infinite sequences in $\mathbb{N_+}^{\mathbb{N}}$.
A point $\alpha \in (0,1)\backslash \mathbb{Q}$ will be often identified with the infinite sequence $\{a_n\}_{n\in \mathbb{N}_+}$ in $\mathbb{N_+}^{\mathbb{N}_+}$.

For convenience, we will also use the following notation:
\bes
[a_0; a_1, \dots, a_n , \dots ] = a_0 + \frac{1}{a_1 + \frac{1}{ \ddots}}, \qquad [a_0; a_1, \dots, a_n ] = a_0 + \frac{1}{a_1 + \frac{1}{\ddots + \frac{1}{a_n}}}.
\ees

\subsection{Natural extension of the Gauss map.}
The natural extension $\h{\G}$ of the Gauss map $\G$ acts on infinite bi-sided sequences $\{a_n\}_{n\in \mathbb{Z}} \in {\mathbb{N}_+}^ {\mathbb{Z}}$
%, $a_n \in \mathbb{N}^+$,
 as the two-sided shift, i.e. $\h{\G} \{a_n\}_{n\in \mathbb{Z}}=  \{ a'_n \}_{n\in \mathbb{Z}}$ where $a'_n = a_{n+1}$. The map $\h{\G}$ admits the following geometric interpretation. Consider the domain $D(\h{\G}) = (0,1)\backslash {\mathbb{Q}} \times (0,1)\backslash \mathbb{Q}$. Let us identify the sequence $ \{a_n\}_{\mathbb{Z}}$ with the point  $\h{\alpha}  = (\h{\alpha} ^-, \h{\alpha}^+) \in D(\h{\G})$, which is given by
\be\label{alpha+-}
\h{\alpha} ^+ = [a_1, a_2, \dots , a_n , \dots ]; \qquad \h{\alpha}^- = [a_0, a_{-1}, \dots , a_{-n} , \dots ].
\ee
Then  $\h{\G}({\h{\alpha}})= \h{\beta}$ where  $\h{\beta}  = (\h{\beta} ^-, \h{\beta}^+)  $ and
\bes
\h{\beta}^+ = \G ( \h{\alpha}^+ ) = \left\{ \frac{1}{\h{\alpha}^+}\right\} = \frac{1}{\h{\alpha}^+} - a_1; \qquad  \h{\beta}^- = \frac{1}{ \left[ \frac{1}{\h{\alpha}^+}\right]+ \h{\alpha}^- } = \frac{1}{ a_1 + \h{\alpha}^- }
\ees
%The sequence $\{a_n\}_{n\in \mathbb{Z}}$ is the symbolic coding of $\h{\alpha} \in D(\h{\G})$ under $\G$ in the sense that $a_n = \left[ \left(\h{\G}^{n-1} \h{\alpha}\right)^{-1} \right]$
Clearly, denoting by $\pi$ the projection $\pi( \h{\alpha}) = \h{\alpha}^+$ or equivalently $\pi( \{a_n\}_{n\in \mathbb{Z}}) =  \{a_n\}_{n\in \mathbb{N_+}}  $,  we have $\pi \h{\G}  =  \G \pi$.
The sequence $\{a_n\}_{n\in \mathbb{Z}}$ is the symbolic coding of $\h{\alpha} \in D(\h{\G})$ under $\h{\G}$ in the sense that 
$a_n = \left[ \left( \pi \h{\G}^{n-1} \h{\alpha}\right)^{-1}  \right]$
%$a_n = \left[ \left(\h{\G}^{n-1} \h{\alpha}\right)^{-1} \right]$

The map $\h{\G}$ admits a natural invariant probability measure $\mu_2 $ on $D(\h{\G})$ which is given by the density: % $\rho_2 = \rho_2 (\alpha^-, \alpha^+)$, 
\bes
 \rho_2 (\alpha^-, \alpha^+) = \frac{1}{\ln 2(1+ \alpha^- \alpha^+)^2 }.%, \qquad \mu_2 =  \rho_2 (\alpha^-, \alpha^+) \ud \alpha^- \ud \alpha^+ .
\ees
\begin{rem}\label{mu2pimu1}
The Gauss measure $\mu_1$ can be recovered as $\pi_* \mu_2$, i.e. for each measurable set $A\subset (0,1)$, we have $\mu_1(A)= \mu_2(\pi^{-1}A)$.
\end{rem}

Given any  $\h{\alpha} = \{a_n\}_{n\in \mathbb{Z}}  \in D(\h{\G})$,  $\h{\alpha}^-$ and  $\h{\alpha}^+$ will always denote the two components of  $\h{\alpha} \in D(\h{\G})$ which are given explicitly in terms of the $a_n$ by (\ref{alpha+-}). 

%Given $\h{\alpha} \in D(\h{\G})$, 
Let $q_n=q_n(\h{\alpha }) = q_n({\h{\alpha}^+ })$, $n\in \mathbb{N}_+$, be the sequence of denominators of the convergents of $\h{\alpha}^+ $. Also, given $R>0$, $n_R(\h{\alpha})$ and  $q_{n_R}(\h{\alpha})$ are set equal to the analogous quantities defined for $\alpha=\h{\alpha}^+ $. 
\begin{rem}\label{n_Rfuture} %Given $\h{\alpha} = \{a_k\}_{k\in \mathbb{Z}}$, 
By construction, the functions $q_n$ (for any $n\in \mathbb{N_+}$),  $n_R$ and $q_{n_R}$ (for any $R>0$)  on $D(\h{\G})$ %depend on $a_k$ with $k\geq 1$ only, hence they
 are constant on fibers $\pi^{-1}\alpha$, $\alpha\in (0,1)\backslash {\mathbb{Q}}$. 
\end{rem}

%\begin{rem}\label{n_Rfuture} Given $\h{\alpha} = \{a_k\}_{k\in \mathbb{Z}}$, both $q_n_R$ and $n_R$ depend on $a_k$ with $k\geq 1$ only, hence they are constant on fibers $\pi^{-1}\alpha$, $\alpha\in (0,1)\backslash {\mathbb{Q}}$. 
%\end{rem}

\subsection{Cylinders.}\label{cylinderssec}
For $b_k \in \mathbb{N}_+$, $k=1, \dots, n$, denote  $\Cp{b_1, \dots, b_n}$ the cylinder 
\bes \Cp{b_1, \dots, b_n} = \{ \alpha=\{a_n\}_{n\in \mathbb{N}_+} \in (0,1)\backslash{\mathbb{Q}} \, : \, a_k = b_k, 1\leq k\leq n \}.
 \ees
We will denote by $\Cnp$ the set of all cylinders  of length $n$, i.e, the set of all $\Cp{b_1, \dots, b_n}$ with  $b_k\in \mathbb{N}_+$ for $1\leq k \leq n$. Moreover, if $\Ci \in \Cnp$, we will denote by $\h{\Ci }$ the set $\pi^{-1}\Ci \subset D(\h{\G}) $.

More generally, given $b_k \in \mathbb{N}_+$, $k=0,\pm 1, \dots, \pm n$, let 
\bes \C{b_{-n},\dots, b_0;  b_1, \dots, b_n} = \{ \alpha=\{a_n\}_{n\in \mathbb{Z}} \in D(\h{\G}) \, : \, a_k = b_k, -n\leq k\leq n \}
 \ees 
and $\Cn$ the set of all bi-sided cylinders  of length $n$, i.e, the set, as  $b_k\in \mathbb{N}_+$ for $-n \leq k \leq n$, 
 of all $\C{b_{-n}, \dots,b_0; b_1, \dots,  b_n}$.
% with  $b_k\in \mathbb{N}_+$ for $-n \leq k \leq n$.
\begin{rem}\label{measureCn}
From the expression of the Gauss density and Remark \ref{mu2pimu1}, we get that $\mu_2(\hC{n})= \mu_1(\C{n})= \mu_1\left( \frac{1}{n+1}, \frac{1}{n} \right)=  O \left( \frac{1}{n^2}\right) $.
\end{rem}

\subsection{Special flows.}\label{suspensionflowsec}
Consider a probability space  $(D, \mathscr{B}, \mu_2)$ and an invertible map $F:D\rightarrow D$ which preserves $\mu_2$.
Let $\varphi: D\rightarrow \mathbb{R}_+$ be a strictly positive function such that $\int_{D} \varphi(\alpha) d\mu_2 < \infty$. 
The {phase space $D_{\Phi}$} of the special flow is the subset of $D\times \mathbb{R}$ given by
\bes
D_{\Phi} = \{ (x,y) |\quad  x \in D  \st 0\leq y < \varphi(x) \}
\ees
and can be depicted as the set of points below the graph of the roof function $\varphi$.
%Let $p(x,y)=x$ be the projection to the base of the special flow. 
Consider the normalized measure $\mu_3$ which is the restriction to $D_{\Phi}$ of the product measure $\left(\int_{D} \varphi(\alpha) d \mu_2 \right)^{-1} \mu_2 \times \lambda$, where $\lambda$ denotes the Lebesgue measure on $\mathbb{R}$. %Clearly $\mu_3$ is invariant under $\Phi_t$ for each $t\in \mathbb{R}$.

The \emph{special flow $\{ \Phi_t \}_{t\in \mathbb{R}}$ built over $F$ with the help of the roof function $\varphi$} is a one-parameter group of $\mu_3$-measure preserving transformations of $D_\Phi$ whose action is generated by the following two relations:
\begin{equation} \label{suspflow}
 \left\{  \begin{array}{lll}
\Phi_t(x,y) &=& (x, y+t), \qquad \mathrm{if} \, 0\leq y+t < \varphi(x); \\
\Phi_{\varphi(x)}(x,0) &=& ( F(x),0).\\
		     \end{array} \right.
\end{equation}
Under the action of the flow a point of $(x,y) \in D_\Phi$  moves with unit velocity along the vertical line up to the point $(x,\varphi(x))$, then jumps instantly to the point $\left( F(x),0 \right)$, according to the base transformation. Afterward it continues its motion along the vertical line until the next jump and so on (see e.g. \cite{CFS:erg}).
Abusing the notation, we will often identify any set $C\subset D$ with $C\times \{ 0 \} \subset D\times \{0\}\subset D_{\Phi}$.

%Let $\BS{\varphi,F}{0}( x) :=\ 0$. 
We will denote 
by
%\footnote{The dependence on $F$ is omitted when there is no ambiguity.}
\bes \BS{\varphi,F}{0}( x) := 0; \qquad
\BS{\varphi,F}{r}( x) = \BS{\varphi}{r}( x)  := \sum_{i=0}^{r-1} \varphi(F^i(x)), \qquad x\in D, \, r\in \mathbb{N}^+, %\qquad \BS{\varphi,F}{0}( x) := 0,%$%,
\ees
the $r^{th}$ non-renormalized \emph{Birkhoff sum} of $\varphi$ along the trajectory of $x$ under $F$. 

Let $t>0$. Given $x\in D$ denote by $r(x,t)$ the integer uniquely defined by
\be \label{defr}
r(x,t):= \min \{ r\in \mathbb{N} \, | \quad  \BS{\varphi}{r}(x) > t \}.
\ee
Then $r(x,t)-1$ gives the number of discrete iterations of $F$ which the point $(x,0)$ undergoes before time $t$.
According to this notation the flow $\Phi_t$ defined by (\ref{suspflow}) acts as
\be \label{flowdef}
\Phi_t(x,0) = \left( F^{r(x,t)-1}(x), t- \BS{\varphi}{r(x,t)-1}(x)\right).
\ee
For $t<0$, the action of the flow is defined as the inverse map.

\section{Reduction to a special flow.}\label{reductionsec}
Let us first define the special flow that we are going to consider.
\subsection{Roof function.}
Consider the following positive real-valued function on $ D(\h{\G})$:
\be \label{varphidef}
\varphi(\h{\alpha}) = \ln \left( a_1 + \frac{1}{a_0 + \frac{1}{a_{-1} + {\dots }}} \right) = -\ln (\h{\G}\h{\alpha})^- .
\ee
The reason for this definition will be clear after Lemma \ref{logqnandBS}.
%Let us remark that $\varphi$ is constant along lines $(\h{\G}^{-1}(\alpha))^- = const $. 
It is easy to see, from  Remark \ref{measureCn}, that 
% to see that  since $\mu_2 ([0,1]\times [\frac{1n,n+1] )= \mu_1 ([n,n+1]) = O(\frac{1}{n})$, 
the function $\varphi $ is integrable with respect to $\mu_2$.

\begin{rem}\label{infsup}
Let $\Ci \in \Cn$, $n\geq 1$, be any cylinder. Then there exists $\delta = \delta(\Ci)>0$ and $M=M(\Ci)>0$ such that $\inf_{\h{\alpha}\in \h{\Ci}} \varphi (\h{\alpha})\geq \delta$ and $\sup_{\h{\alpha}\in \h{\Ci}} \varphi (\h{\alpha})\leq M$. It follows by remarking that $\varphi(\h{\alpha}) \geq \ln (1+\frac{1}{a_0+1})$ and  $\varphi(\h{\alpha}) \leq \ln (a_1 + 1)$.
\end{rem}
Let us consider the special flow $\{ \Phi_t \}_{t\in \mathbb{R}}$ built over $\h{\G}$ under the function $\varphi$ and let $\mu_3 = (\int \varphi \ud \mu_2)^{-1} \mu_2 \times \lambda$ be the  $\Phi_t$-invariant  probability  measure. 
%The main property of the flow $\{ \Phi_t \}_t$ that we use is mixing. 
Let us recall that $\{ \Phi_t \}_{t\in \mathbb{R}}$ is said to be mixing if, for all Borel subsets $A, B$ of $D_{\Phi}$, we have 
\bes
\lim_{t\rightarrow \infty} \mu_3 \left( \Phi_{-t} (A) \cap B \right) = \mu_3(A)\mu_3(B). 
\ees
\begin{prop}\label{mixing}
The flow $\{ \Phi_t \}_t $ is mixing.
\end{prop}
\noindent Proposition \ref{mixing} was proved in \cite{DS:sta}. We recall the proof in \S \ref{mixingsection}.

\subsection{Denominators growth and Birkhoff sums.}
%Given $\h{\alpha} = \{a_k\}_{k\in \mathbb{Z}}\in D(\h{\G})$, let $q_n=q_n(\h{\alpha }) = q_n({\h{\alpha}^+ })$, $n\in \mathbb{N}$, be the sequence of denominators of the convergents of $\h{\alpha}+ $. 
%\begin{rem}\label{n_Rfuture} %Given $\h{\alpha} = \{a_k\}_{k\in \mathbb{Z}}$, 
%Both $q_n_R$ and $n_R$ depend on $a_k$ with $k\geq 1$ only, hence they are constant on fibers $\pi^{-1}\alpha$, $\alpha\in (0,1)\backslash {\mathbb{Q}}$. 
%end{rem}
%Given $\alpha= \{ a_n\}_{n\in \mathbb{N}_+} \in (0,1) \backslash \mathbb{Q}$, let $q_n=q_n(\alpha )$, $n\in \mathbb{N}$, be the sequence of denominators of its convergents.  By abusing the notation, let us also right 
% Pick any $\h{\alpha}$ such that $\h{\alpha}^+=\alpha$.
 Let us show that $\ln q_n$ can be approximated by Birkhoff sums of $\varphi$. % at $\h{\alpha}$.
%\begin{rem}\label{n_Rfuture} Given $\h{\alpha} = \{a_k\}_{k\in \mathbb{Z}}$, both $q_n_R$ and $n_R$ depend on $a_k$ with $k\geq 1$ only, hence they are constant on fibers $\pi^{-1}\alpha$, $\alpha\in (0,1)\backslash {\mathbb{Q}}$. 
%\end{rem}
% Let $\h{\alpha}$  in $D(\h{\G})$ be the point determined by the bi-sided sequence  $\{ a'_n\}_{n\in \mathbb{Z}}$ with $a'_n = a_n$ for $n\geq 1$ and $a'_n= 1 $ for $n\leq 0$. In particular $\pi \h{\alpha} =\alpha^+= \alpha$. Let us show that $\ln q_n$ can be approximated by Birkhoff sums of $\varphi$. 
Indeed, put
\be \label{deffn}
f_n(\h{\alpha}) = \ln q_n (\h{\alpha}) - \BS{\varphi}{n} (\h{\alpha}).
\ee
\begin{lemma}\label{logqnandBS}
There exist a function $f$ on $D(\h{\G})$ such that $f_n $ converges to $f$ uniformly in $\h{\alpha}$ and exponentially fast in $n$, i.e. 
\be \label{errorf}
\ln q_n (\h{\alpha}) = \BS{\varphi}{n} (\h{\alpha}) + f (\h{\alpha}) + \epsilon_n(\h{\alpha}), \qquad  \sup_{\h{\alpha}\in D(\h{\G})} \epsilon_n (\h{\alpha})  = O( 2^{- n}). 
\ee
\end{lemma}
%\{Proof.}
\begin{proof}
Let  $q_0=1$, $q_{-1}=0$ and $r_n = \frac{q_n}{q_{n-1}}$ for $n\in \mathbb{N}_+$,  so that $q_n = \Pi_{k=1}^{n} r_k$. From the well-known relation  $q_{k+1} = a_{k+1} q_k + q_{k-1}$ which holds for $k\geq 0$ (see e.g. \cite{Kh:con}), we get by recursion that, for $k\geq 1$, 
\bes%\label{rrecursion}
r_{k+1}= a_{k+1}+ \frac{1}{r_k} = [a_{k+1}; a_k , \dots, a_1] 
\ees
and hence 
\be \label{logqn}
\log q_n = \sum_{k=1}^n \ln[a_{k}; a_{k-1} , \dots, a_1]  . 
\ee

Since $(\h{\G}^k \h{\alpha})^- = [a_k, a_{k-1}, \dots ]$, from the definition (\ref{varphidef}) of $\varphi$ we get
\be \label{BS}
\BS{\varphi}{n}(\h{\alpha}) = \sum_{k=1}^n \ln \frac{1}{(\h{\G}^k \h{\alpha})^-} = \sum_{k=1}^{n}\ln [a_{k}; a_{k-1}, \dots ].
\ee 
Thus, from (\ref{deffn}) and (\ref{logqn}),
\bes
({f}_{k+1} - {f}_{k} )(\h{\alpha}) =  \ln r_{k+1} - \varphi(\h{\G}^{k} \h{\alpha}) =  \ln r_{k+1} - \ln \frac{1}{(\h{\G}^{k+1} \h{\alpha})^-} = \ln \frac{[a_{k+1}; a_k, \dots, a_1 ]}{[a_{k+1}; a_k,\dots ]}. %= \ln \frac{[a_{n+1}, a_n,\dots, a_1 ]}{[a_{n+1}, a_n, \dots ]} .
\ees
In order to estimate the last term, consider $\beta_k = [ a_k, a_{k-1},\dots ]$ and and let $\left\{ \frac{p_{k,m}}{q_{k,m}}\right\}_m$ be the convergents of $\beta_k$, so that in particular $\frac{p_{k,k}}{q_{k,k}}=[ a_k, \dots, a_1 ] $. Recalling the well-known formula (see \cite{Kh:con})
\be\label{bestapprox}
\left| \beta_k-\frac{p_{k,m}}{q_{k,m}} \right| \leq \frac{1}{(q_{k,m})^2} %\qquad \left| \beta''-\frac{p''_n}{q''_n} \right| \leq \frac{1}{(q''_n)^2};
\ee
and using that for all sequences of denominators of convergents $q_{k,m} \geq 2 ^{\frac{m-1}{2}}$ (\cite{Kh:con}, Thm. 12) and that $[a_{k+1}; a_k,\dots ] \geq 1$, we get
\be \label{exponential}
 \left|\ln 
\frac{[a_{k+1}; a_k, \dots, a_1 ]}{[a_{k+1}; a_k,\dots ]}
%\frac{[a_{k+1}; a_k,a_{k-1}, \dots ]}{[a_{k+1}; a_k, \dots, a_1 ]} 
 \right| \leq \left|\ln \left( 1 + \frac{ \frac{p_{k,k}}{q_{k,k}}- \beta_k}{[a_{k+1}; a_k, \dots ]} \right)  \right|\leq 2 \left| \beta_k-\frac{p_{k,k}}{q_{k,k}} \right| \leq 2^{2-k}.
\ee
Hence, (\ref{exponential}) shows that we can well define
\be \label{fdef}
 {f} (\h{\alpha}) = %{f}_{0} + 
\sum_{k=0}^{\infty} ({f}_{k+1} - {f}_{k} )(\h{\alpha})
\ee
and (\ref{errorf}) clearly follows from the geometric bound of the series, with $\epsilon_n = 2^{3-n}$.
%\ep\medbreak
\end{proof}
\begin{lemma}\label{variationflemma}
If ${\alpha}_1,{\alpha}_2 \in \C{a_{-n},\dots, a_0; a_1, \dots , a_n}$, then 
\be  \label{variationf}
| f (\h{\alpha}_1) -   f (\h{\alpha}_2) |\leq C 2^{- n}, %e^{-cn} .
\ee
where $C>0$ is an absolute constant.
\end{lemma}
\begin{proof}
First, by Lemma \ref{logqnandBS}, we have 
\be \label{difference} 
|{f}(\h{\alpha}_1)- {f}(\h{\alpha}_2 )|\leq   C 2^{-n} +  |{f_n}(\h{\alpha}_1)- {f_n}(\h{\alpha}_2 )|. 
%|{f}(\h{\alpha}_1)- {f_n}(\h{\alpha}_1 )| + |{f_n}(\h{\alpha}_2)- {f}(\h{\alpha}_2 )| \leq 2^{4-n}.
\ee
Let us estimate the second term of the right hand side.
Let $\h{ \alpha}_1  = \{a'_m\}_{m\in \mathbb{Z}}$, $\h{\alpha}_2 = \{a''_m\}_{m\in \mathbb{Z}}$, where by assumption $a'_m=a''_m=a_m$ for $-n\leq m\leq n$.  %Remark that $q_n(\h{ \alpha}_1^+)= q_n(\h{ \alpha}_2^+)$.
 From (\ref{BS}),
\be\label{BSdifference}
  \BS{\varphi}{n}(\h{\alpha}_1)- \BS{\varphi}{n}(\h{\alpha}_2 ) = \sum_{k=1}^{n} \ln \frac{ [a_{k}; a_{k-1}, \dots, a_{-n}, a'_{-n-1}, a'_{-n-2} \dots  ]}{ [a_{k}; a_{k-1}, \dots, a_{-n}, a''_{-n-1}, a''_{-n-2} ,\dots ]}.
\ee 
Arguing as in Lemma \ref{logqnandBS}, for $k=1, \dots, n$, consider 
%$\beta'_k= [a_k, \dots, a_{-n}, a'_{-n-1}, \dots]$ and $\beta''_k= [a_k, \dots, a_{-n}, a''_{-n-1}, \dots]$
\bes
\beta'_k= [a_{k-1},a_{k-2}, \dots, a_{-n}, a'_{-n-1}, \dots] \quad \mathrm{and}\quad  \beta''_k= [a_{k-1},a_{k-2}, \dots, a_{-n}, a''_{-n-1}, \dots] 
\ees
 and denote  by $\left\{\frac{ p^{'}_{k,n}}{q^{'}_{k,n}}\right\}_n$ and $\left\{ \frac{p^{''}_{k,n}}{q^{''}_{k,n}}\right\}_n$ their respective convergents. Since $\frac{p^{'}_{k,k+n}}{q^{'}_{k,k+n}} =\frac{p^{''}_{k,k+n}}{q^{''}_{k,k+n}}$,  from (\ref{bestapprox}) and $q_{m} \geq 2 ^{\frac{m-1}{2}}$ we get
\bes
\left|\beta'_k - \beta''_k \right| \leq \left|\beta'_k - \frac{p^{'}_{k,k+n}}{q^{'}_{k,k+n}}\right| +  \left|\beta''_k - \frac{p^{''}_{k,k+n}}{q^{''}_{k,k+n}}\right| \leq \frac{1}{(q^{'}_{k,k+n})^2} + \frac{1}{{(q^{''}_{k,k+n})} ^2} \leq 2^{2-k-n}.
\ees
%, we get,
Hence, as in (\ref{exponential}),
%since $p^{k,'}_{k+n+1}/q^{k,'}_{k+n+1} =p^{k,''}_{k+n+1}/q^{k,''}_{k+n+1}$, reasoning as in (\ref{exponential}), using (\ref{bestapprox}), 
\bes
 \left| \sum_{k=1}^{n} \ln \frac{ [a_{k}; \dots, a_{-n}, a'_{-n-1}, \dots  ]}{ [a_{k}; \dots, a_{-n}, a''_{-n-1}, \dots ]}  \right| \leq   \sum_{k=1}^{n}
% \left|\ln \left( 1 + \frac{\beta_k - \frac{p'_k}{q'_k}}{[a_{k+1}; a_k, \dots , a_1]} \right)  \right|\leq
 2 \left| \beta'_k-\beta''_k \right| \leq  \sum_{k=1}^{n}  2^{3-k-n} \leq C 2^{-n}.
\ees
%where the last inequality follows from (\ref{bestapprox}) using that  $p^{k,'}_{k+n+1}/q^{k,'}_{k+n+1} =p^{k,''}_{k+n+1}/q^{k,''}_{k+n+1}$ (\ref{bestapprox}),
Hence,  by (\ref{deffn}), remarking that we also have $q_n(\h{ \alpha}_1^+)= q_n(\h{ \alpha}_2^+)$, this implies that $ |{f_n}(\h{\alpha}_1)- {f_n}(\h{\alpha}_2 )| \leq C 2^{-n}$ and gives the desired estimate of (\ref{difference}).
\end{proof}
%\ep \medbreak

%To prove the second part of the Lemma, let us remark that given $\h{\alpha}= \{a_m\}_{m\in \mathbb{Z}}$, the value of ${f}_{k}(\h{\alpha})$ depends only on $a_m$ with $m\leq k$.
%Since $\alpha_1, \alpha_2 \in \C{a_1, \dots , a_n}$, by construction of the lifts  $\h{ \alpha}_1  = \{a'_m\}_{m\in \mathbb{Z}}$ and $\h{\alpha}_2 = \{a''_m\}_{m\in \mathbb{Z}}$, we have $a'_m=a''_m$ for $m\leq n$. Hence, we have in particular 
%\bes {f}_{k}(\h{\alpha}_1) =  {f}_{k}(\h{\alpha}_2 ), \qquad \forall k \leq n\ees
%$ {f}_{n}(\h{\alpha}_1) =  {f}_{n}(\h{\alpha}_2) $
%and,  using the first part of the Lemma, 
%\bes
%|{f}(\h{\alpha}_1)- {f}(\h{\alpha}_2 )|\leq  |{f}(\h{\alpha}_1)- {f_n}(\h{\alpha}_1 )| + |{f_n}(\h{\alpha}_2)- {f}(\h{\alpha}_2 )| \leq 2^{4-n}.
%\ees
%\ees
%and, from definition (\ref{fdef}) of $f$, by (\ref{exponential}), we get
%\bes
%{f}(\h{\alpha}_1)- {f}(\h{\alpha}_2 ) = \sum_{k=n-1}^{\infty}({f}_{k+1} - {f}_{k})
%\ees
%the difference $({f}_{k+1} - {f}_{k} )(\h{\alpha})$ depends only on $a_m$ with $m\leq k+1$ and ${f}_{0} (\h{\alpha})$ depends on $a_m$ with $m\leq 1$. Since $\alpha_1, \alpha_2 \in \C{a_1, \dots , a_n}$, by construction of the lifts  $\h{ \alpha}_1  = \{a'_m\}_{m\in \mathbb{Z}}$ and $\h{\alpha}_2 = \{a''_m\}_{m\in \mathbb{Z}}$, we have $a'_m=a''_m$ for $m\leq n$. Hence, we have
%\bes({f}_{k+1} - {f}_{k})(\h{\alpha}_1) =  ({f}_{k+1} - {f}_{k})(\h{\alpha}_2 ), \qquad \forall k+1 \leq n
%\ees
%and, from definition (\ref{fdef}) of $f$, by (\ref{exponential}), we get
%\bes
%{f}(\h{\alpha}_1)- {f}(\h{\alpha}_2 ) = \sum_{k=n-1}^{\infty}({f}_{k+1} - {f}_{k})
%\ees

\subsection{Comparing renewal times.}% Let $\h{\alpha}$ be fixed.
Given $\h{\alpha}$ and $R>0$, we want to choose  $T$ as a function of $R$ so that we can compare $n_R(\h{\alpha})$ and $r(\h{\alpha}, T)$, where $r(\h{\alpha}, T)-1$ is the discrete number of iterations undergone by $\Phi_t (\h{\alpha},0)$ for $t\leq T$, see (\ref{defr}).
Let us recall that $n_R(\h{\alpha})$ is uniquely determined by
\bes
\ln q_{n_R(\h{\alpha})-1} \leq  \ln R < \ln q_{n_R(\h{\alpha})}. 
\ees
By  (\ref{deffn}), the previous inequality can be rewritten as
\be \label{n_Rinequality}
 \BS{\varphi}{n_R(\h{\alpha})-1} (\h{\alpha})+ f_{n_R(\h{\alpha})-1}(\h{\alpha}) \leq \ln R <  \BS{\varphi}{n_R(\h{\alpha})} (\h{\alpha}) + f_{n_R(\h{\alpha})}(\h{\alpha}).
\ee
%Hence, motivated by (\ref{n_Rinequality}), we would like to consider $T= \ln R - f(\h{\alpha})$. 
To avoid the dependence of the time on $\alpha$, let us localize to a set of $\Ci \subset D(\h{\G})$ and denote $f_{\Ci} = \sup_{\h{\alpha} \in \Ci } f(\h{\alpha})$. Assume that for all $\h{\alpha}\in \Ci$ we have $| f (\h{\alpha}) -  f_{\Ci}|\leq \epsilon/2$ (such sets will be constructed in the Proof of Theorem \ref{main}).

%To avoid the dependence on $\h{\alpha}$, given $\epsilon > 0$ consider any localized set $\Ci \in \Cn $, where $n$ is chosen large enough so that, by (\ref{variationf}) of Lemma \ref{logqnandBS}, we have $| f (\h{\alpha}_1) -   f (\h{\alpha}_2) |\leq \epsilon$ for all $\alpha_1, \alpha_2 \in \h{\Ci}$. Let $f_{\Ci} = \sup_{\alpha \in \Ci } f(\h{\alpha})$ and set $T=T_{\Ci} = \ln R + f_{\Ci}$.
%Let %$U=\{\h{\alpha}  \in D(\h{\G}) \st n_R \neq r(T)$
%\bes
%U=\{\h{\alpha}  \in \h{\Ci} \st n_R \neq r(T).
%\ees

Let us first show that on large measure sets the growth of $n_R$ is guaranteed by the growth of $R$.
\begin{lemma}\label{badpoints}
For each measurable $\Ci\subset D(\h{\G})$ and $\epsilon>0$ there exits a measurable $ \Ci '\subset \Ci$ %and $c=c_{\Ci}>0$ 
such that $\mu_2(\Ci\backslash \Ci')\leq \epsilon \mu_2(\Ci)$ and $\min_{\h{\alpha}\in \Ci'}  n_R(\h{\alpha} ) $ tends to infinity uniformly as $R$ tends to infinity.

Similarly, given  $\epsilon>0$ there exits a measurable $ \Ci_{\epsilon} \subset (0,1)$ %and $c=c_{\Ci}>0$ 
such that $\mu_1((0,1)\backslash \Ci_{\epsilon})\leq \epsilon $ and $\min_{{\alpha}\in \Ci_{\epsilon}}  n_R({\alpha} ) $ tends to infinity uniformly as $R$ tends to infinity.
\end{lemma}
\begin{proof}
%\proc{Proof.}
By L{\'e}vy-Khinchin Theorem, for $\mu_1$-a.e. $\alpha \in (0,1)$, there exists an absolute constant $l>0$ such that $\lim_{n\rightarrow \infty} \frac{\ln q_n}{n} =  l$. %(see \cite{Ki:con}, Theorem 31; a stronger assertion is that the previous limit exist).
 By Remarks \ref{mu2pimu1} and \ref{n_Rfuture}, the same holds for  $\mu_2$-a.e. $\h{\alpha }\in D(\h{\G})$. By Egorov's theorem, we can find for each $\epsilon>0$ a measurable subset $\Ci_1\subset \Ci$ with $\mu_2(\Ci\backslash \Ci_1 )\leq \frac{\epsilon}{2} \mu_2(\Ci)$ on which the convergence is uniform, so that for some $\overline{n}$, $\frac{\ln q_n}{n} \leq 2 l$ for each $n\geq \overline{n}$. Moreover, there exists $\Ci_2\subset \Ci$ with  $\mu_2(\Ci\backslash \Ci_2 )\leq \frac{\epsilon}{2} \mu_2(\Ci)$ such that on $\Ci_2$, the functions $\frac{\ln q_n}{n} $ for $n=0,\dots, \overline{n}$ are uniformly bounded.  Hence, setting $\Ci'= \Ci_1\cap \Ci_2$, $\mu_2(\Ci\backslash \Ci' )\leq \epsilon \mu_2(\Ci)$ and there exists a constant $c=c(\Ci , \epsilon)>0$ such that for all $\h{\alpha}\in \Ci' $ and all $n\in \mathbb{N}_+$ we have $\ln q_{n} \leq c n $. Since by definition $q_{n_R(\h{\alpha})}>R$, this implies that $\min_{\h{\alpha}\in \Ci'}  n_R(\h{\alpha}) \geq (c)^{-1} \ln q_{n_R(\h{\alpha})}  \geq (c)^{-1} \ln R $, from which the Lemma follows. 
The proof of the second part proceeds in exactly the same way.
\end{proof}
%\ep \medbreak

\begin{lemma} \label{comparing}
Assume that $\Ci\in \mathscr{C}_n$ and for all $\h{\alpha}\in \Ci$ we have $| f (\h{\alpha}) -  f_{\Ci}|\leq \epsilon/2$. There exists $R_0 =R_0(\Ci) >0$ such that, whenever $R\geq R_0$, if we set $T=T(R, \Ci)=  \ln R - f_{\Ci}$ and consider
\bes
U = U_{\Ci}: =\{\h{\alpha}  \in \Ci 
\st n_R(\h{\alpha}) \neq r(\h{\alpha}, T) \},
\ees
%Then, for some %$R=R_0(\epsilon, \Ci)$,  if $R\geq R_0$, 
% $n_0=n_0(\epsilon)$, if $n_R\geq n_0$
 we have $\mu_2(U \cap \Ci)\leq 7 \epsilon \mu_2(\Ci)$. 
\end{lemma}
\noindent Hence, outside of a subset of $\Ci$ of arbitrarily small proportion and for large $R$, the function $T=T(R,\Ci)$ is such that $ n_R(\cdot) = r(\cdot, T)$.
%\proc{Proof.}
\begin{proof}
 Let $T=T(R, \Ci)=  \ln R - f_{\Ci}$.
 %For brevity, let us denote $r(\h{\alpha}, T)(\h{\alpha})$ by $r(\h{\alpha}, T)$. 
By definition, we have 
%$r(\h{\alpha}, T)(\h{\alpha})=r(\h{\alpha}, T)$ we have
\be \label{r(T)inequality}
 \BS{\varphi}{r(\h{\alpha}, T)-1} (\h{\alpha})\leq  T = \ln R - f_{\Ci}  <  \BS{\varphi}{r(\h{\alpha}, T)} (\h{\alpha}).
\ee
Let $\Ci'\subset \Ci$ be given by Lemma \ref{badpoints}.  
Let us show, by comparing (\ref{r(T)inequality}) to (\ref{n_Rinequality}), that,  as long as $R\geq R_0$ for some $R_0$ defined below, we have  
%$U \cap \Ci ' \subset U_{\epsilon} \cup U_{-\epsilon} \cap \Ci '  $
\be\label{inclusion}
U \cap \Ci ' \subset U_{\epsilon} \cup U_{-\epsilon} \cap \Ci ' 
\ee
 where  $U_{\pm \epsilon}\subset D(\h{\G})$ are defined as   
\be \label{Uepsilondef}
 U_{-\epsilon} = \{ \h{\alpha} %\in \Ci
  \st T < \BS{\varphi}{r(\h{\alpha}, T)}(\h{\alpha}) \leq T + \epsilon \}; \quad  U_{\epsilon} = \{ \h{\alpha} %\in \Ci 
 \st T - \epsilon < \BS{\varphi}{r(\h{\alpha}, T)-1}(\h{\alpha}) \leq T \}. 
\ee
%Remark that if $\varphi$
To show the inclusion (\ref{inclusion}), assume that $ \h{\alpha}\in \Ci '$, but $ \h{\alpha}\notin U_{\epsilon} \cup U_{-\epsilon}$. Then, by (\ref{Uepsilondef}) and (\ref{r(T)inequality}), we have 
\be \label{strictr(T)inequalities}
 \BS{\varphi}{r(\h{\alpha}, T)-1} (\h{\alpha})\leq T-\epsilon; \qquad \BS{\varphi}{r(\h{\alpha}, T)}(\h{\alpha}) > T + \epsilon .
\ee
Choose $n_0\gg 1$ so that by Lemma \ref{logqnandBS}, for all $n\geq n_0-1$, $|f_n - f|\leq \epsilon/2$. By {Lemma \ref{badpoints}} we can choose $R_0$ so that, if $R\geq R_0$, then $n_R(\h{\alpha})\geq n_0$ for all $\h{\alpha}\in \Ci'$.
 This, together with the  assumptions on $\Ci$, implies that, since $\h{\alpha}\in \Ci'$,   $|f_{i}(\h{\alpha}) -  f_{\Ci}|<\epsilon$ for $i=n_R(\h{\alpha}), i=n_R(\h{\alpha})-1$. Hence, (\ref{n_Rinequality}) gives
\bes
 \BS{\varphi}{n_R(\h{\alpha})}  > \ln R -  f_{n_R(\h{\alpha})}(\h{\alpha}) > T -\epsilon; \quad
 \BS{\varphi}{n_R(\h{\alpha})-1} (\h{\alpha}) \leq \ln R - f_{n_R(\h{\alpha})-1}(\h{\alpha})  < T+ \epsilon,
\ees
which, compared with (\ref{strictr(T)inequalities}), since  $\BS{\varphi}{n}(\h{\alpha})$ is increasing in $n$, implies that 
\bes
r(\h{\alpha}, T)-1< n_R(\h{\alpha}) \quad \mathrm{and} \quad n_R(\h{\alpha})- 1 < r(\h{\alpha}, T) .
\ees
Since both $n_R(\h{\alpha})$ and $r(\h{\alpha}, T)$ are integers, these inequalities imply $n_R(\h{\alpha})=r(\h{\alpha}, T)$, i.e, $\h{\alpha}\notin U$, proving (\ref{inclusion}) as desired.

Recalling the definition (\ref{flowdef}) of the special flow action, we can rewrite the sets $U_{\pm\epsilon}$ as
\begin{eqnarray}\nonumber
%&&
 U_{\epsilon} &= &\{ \h{\alpha}  %\in C 
 \st 0 \leq T- \BS{\varphi}{r(\h{\alpha}, T)-1} < \epsilon \} = \{ (\h{\alpha}, 0)  %\in \Ci\times\{0\}
 \st \Phi_{T}\left(\h{\alpha}, 0\right) \in D_{\Phi}^{\epsilon} \} ,\\
%&&
 U_{-\epsilon} &= &\{ \h{\alpha}  %\in \Ci  
\st  \varphi(\h{\G}^{r(\h{\alpha}, T)-1}\h{\alpha})-\epsilon \leq  T- \BS{\varphi}{r(\h{\alpha}, T)-1} <   \varphi(\h{\G}^{r(\h{\alpha}, T)-1}\h{\alpha}) \} = \nonumber \\ &=&\{ (\h{\alpha}, 0)  
%\in \Ci\times\{0\} 
\st \Phi_{T}\left(\h{\alpha}, 0\right)  \in  D_{\Phi}^{-\epsilon}  \},
\nonumber
\end{eqnarray}
where $D_{\Phi}^{\epsilon} = D(\h{\G})\times [0, \epsilon)$ and $D_{\Phi}^{-\epsilon} = \{ (\h{\alpha}, y) \st \varphi(\h{\alpha})-\epsilon \leq y <  \varphi(\h{\alpha})\}$.

We want to use mixing of $\{\Phi_t\}_{t\in \mathbb{R}}$ to estimate the measures of the last two sets. In order to do this, we need to ``thicken'' them as follows. Choose  $0< \delta \leq \epsilon$ such that, by Remark \ref{infsup}, $\delta < \min_{\h{\alpha}\in \Ci} \varphi(\h{\alpha})$ 
%Let $\Ci_{\delta} = \Ci\times [0,\delta)$ and
and  consider the following two subsets of $D_{\Phi}$:  %$ U_{\pm\epsilon}^\delta \subset D_{\Phi}$ 
\bes
 U_{\pm\epsilon}^\delta = \{ (\h{\alpha}, z) \st 0\leq z < \delta;\,  \Phi_{T}\left(\h{\alpha}, z\right) \in D_{\Phi}^{\pm \epsilon } \}=  D_{\Phi}^{\delta }  \cap  \Phi_{-T}  D_{\Phi}^{\pm \epsilon } . %\subset D_{\Phi} .%; \qquad 
%U_{\epsilon}^\delta = \{ (\h{\alpha}, z) \st 0\leq z < \delta;\,  \Phi_{T}\left((\h{\alpha}, z)\right) \in D_{\Phi}^{\epsilon} \};
\ees
%It is easy to check that, If
Let us show that if  $\h{\alpha}\in U_{\epsilon}\cap \Ci$, then for each $0\leq z < \delta$, we have $(\h{\alpha}, z) \in  U_{\epsilon+\delta}^\delta $. Indeed, by choice of $\delta$, 
%$\Phi_T(\h{\alpha},\delta) = \Phi_{T+\delta}(\h{\alpha},0) = \Phi_{\delta} (\G^{r(T)-1 \h{\alpha}}, T - \BS{\varphi}{r(T)-1} (\h{\alpha}))$
\be\label{cases} \begin{split}
& \Phi_T(\h{\alpha},z) = \Phi_{T+z}(\h{\alpha},0) = \Phi_{z} (\h{\G}^{r(\h{\alpha}, T)-1} \h{\alpha}, T - \BS{\varphi}{r(\h{\alpha}, T)-1} (\h{\alpha}))=\\&=\left\{\begin{array}{l} \left( \h{\G}^{r(\h{\alpha}, T)-1} \h{\alpha}, T+z  - \BS{\varphi}{r(\h{\alpha}, T)-1} (\h{\alpha})\right) \quad 
%& %\mathrm{if} \, 
\left\{ \begin{array}{l}\mathrm{if} \,\, \BS{\varphi}{r(\h{\alpha}, T)}(\h{\alpha}) > T+z ;
% - \BS{\varphi}{r(\h{\alpha}, T)-1} < \varphi(\h{\G}^{r(\h{\alpha}, T)-1} \h{\alpha}) \,
\\ \mathrm{and}  \, r(\h{\alpha}, T)= r( \h{\alpha}, T+z);  \end{array} \right. \\ 
\left(\h{\G} ^{r-1} \h{\alpha} , T+z  - \BS{\varphi}{r-1} (\h{\alpha}) \right)
% & %\mathrm{if} \,
\quad
\left\{ \begin{array}{l} \mathrm{if} \,\, \BS{\varphi}{r(\h{\alpha}, T)}(\h{\alpha}) \leq  T+z 
% - \BS{\varphi}{r(\h{\alpha}, T)-1} \geq \varphi(\h{\G}^{r(\h{\alpha}, T)-1 }\h{\alpha}) \,
\\ \mathrm{and} \, r=r(\h{\alpha}, T+z ) >  r(\h{\alpha}, T) .\end{array} \right. 
\end{array} \right.
     \end{split}
\ee
Since $r> r(\h{\alpha}, T)$ and $\BS{\varphi}{r}$ is increasing in $r$ and  by definition of $U_{\epsilon}$, both  $ T+z  - \BS{\varphi}{r-1} \leq T+\delta  - \BS{\varphi}{r(\h{\alpha}, T)-1} (\h{\alpha})\leq \epsilon + \delta$, hence  $(\h{\alpha}, z) \in  U_{\epsilon+\delta}^\delta  $. % and  $T+\delta  - \BS{\varphi}{r-1} $

Reasoning in a similar way, let us also show that if  $\h{\alpha}\in U_{-\epsilon}\cap \Ci$, then for each $0\leq z < \delta$, we have $(\h{\alpha}, z) \in   U_{-\epsilon}^\delta \cup U_{\delta}^\delta $. Indeed, from (\ref{cases}), in the first case $T+z -\BS{\varphi}{r(\h{\alpha}, T)-1} < \varphi(\h{\G}^{r(\h{\alpha}, T)-1}) (\h{\alpha}) $ and since  $\h{\alpha}\in U_{-\epsilon}$, also $T+z -\BS{\varphi}{r(\h{\alpha}, T)-1}  \geq T -\BS{\varphi}{r(\h{\alpha}, T)-1} \geq \varphi(\h{\G}^{r(\h{\alpha}, T)-1}) (\h{\alpha})  -\epsilon $, so    $(\h{\alpha}, z) \in   U_{-\epsilon}$;  while in the second case, since $r-1 \geq r(\h{\alpha}, T)$, we have $\BS{\varphi}{r}(\h{\alpha})\geq \BS{\varphi}{r(\h{\alpha}, T)}(\h{\alpha}) > T$ and hence  $0\leq T+z  -\BS{\varphi}{r}  < z < \delta$, so  $(\h{\alpha}, z) \in   U_{\delta}$. 

Let $\Ci'_{\delta} = \Ci' \times [0,\delta)$. Hence, recalling also (\ref{inclusion}), we proved that 
\bes
(U\cap \Ci') \times [0,\delta) \subset ((U_{\epsilon} \cup U_{-\epsilon })\cap \Ci' )\times [0,\delta) \subset \Ci'_\delta \cap  U_{-\epsilon}^\delta \cup U_{\delta+\epsilon}^\delta = \Ci'_{\delta}\cap \Phi_{-T} (D_{\Phi}^{-\epsilon} \cup D_{\Phi}^{\delta+\epsilon} ) .
\ees
Considering the measures of the above sets and using mixing (Proposition \ref{mixing}), one can find $T_0$ such that, as soon as $T\geq T_0$, %(which can be achieved enlarging $R_0$ if necessary, since $T_0 = \ln R_0 -f_{\Ci}$)
 we have
\be\label{mixingcons}
 \begin{split}
\mu_2(U\cap \Ci')\delta  = 
\mu_3 (U\cap \Ci' \times [0,\delta) )& \leq \mu_3 (\Ci'_{\delta}\cap \Phi_{-T} (D_{\Phi}^{-\epsilon} \cup D_{\Phi}^{\delta+\epsilon} ) ) \leq  \\ &
\leq 2\mu_2(\Ci')\delta \mu_3( D_{\Phi}^{-\epsilon} \cup D_{\Phi}^{\delta+\epsilon}  )  \leq 2 \mu_2(\Ci') \delta (3\epsilon) 
,    \end{split}
\ee
where the last inequality follows from the fact that $\mu_3( D_{\Phi}^{\pm\epsilon} )\leq \epsilon$ and gives $\mu_2(U\cap \Ci') \leq 6\epsilon \mu_2(\Ci')$. Enlarging $R_0$ if necessary so that $\ln R_0 -f_{\Ci}\geq T_0$, if $R\geq R_0$ also $T\geq T_0$ and (\ref{mixingcons}) holds. Hence  
%which can be achieved enlarging $R_0$ if necessary, since $T_0 = \ln R_0 -f_{\Ci}$)
 $\mu_2(U\cap \Ci)\leq  \mu_2(U\cap \Ci\backslash \Ci') + 6\epsilon  \mu_2(\Ci) \leq 7\epsilon \mu_2(\Ci)$, concluding the proof of the Lemma.
\end{proof}
% \ep \medbreak

\section{Existence of the limiting distribution.}\label{limitingsec}
\begin{proofof}{Theorem}{main}
%\proc{Proof of Theorem \ref{main}.}
Assume $b> a >1$ and $c_k\in \mathbb{N}_+$, $0\leq k< N$. We want to estimate the expression (\ref{limiting}). 
% Remark that, given $\h{\alpha} = \{a_k\}_{k\in \mathbb{Z}}$, both $q_n_R$ and $n_R$ depend on $a_k$ with $k\geq 1$ only, hence they are constant on $\pi^{-1}\alpha$. 
Recalling Remark \ref{n_Rfuture}, as soon as $n_R({\alpha})>N$ 
we can rewrite the condition $ a_{n_R({\alpha})-k} = c_k, \, 0\leq k < N $ in an equivalent way as 
\be\label{conditionona}
\h{\G}^{n_R(\h{\alpha})-1}\h{\alpha} \in C_N, \qquad \mathrm{where} \quad \h{\alpha}^+=\alpha, \quad C_N := \h{\G}^{N-1} \hC{c_{N-1,} c_{N-2} \dots, c_0 },
\ee
since if (\ref{conditionona}) holds,  $\{ a'_j \}_{j\in \mathbb{Z}} :=  \h{\G}^{n_R(\h{\alpha})-N} \h{\alpha} = \h{\G}^{-(N-1)} \h{\G}^{n_R(\alpha)-1}\h{\alpha}\in \hC{ c_{N-1} \dots, c_0 } $, so  $a_k'= c_{N-k}$ for $1\leq j \leq N$ and  $a'_j = a_{n_R(\alpha)-N+ j}$ by definition of $\h{\G}$, %. and $a_k'= c_{N-k}$ by the condition above, so 
hence $a_{n_R(\alpha)-N+ j} = c_{N-k}$ for $1\leq j \leq N$ gives the desired set of equalities for $k=N-l$.

Given two functions $g_1, g_2$ on $D(\h{\G})$, $g_1\leq g_2$, let us denote by $D_{\Phi}(g_1,g_2)$ the following subsets: 
\bes
D_{\Phi}(g_1,g_2) = \{ (\h{\alpha}, y) \in D_{\Phi} \st \varphi(\h{\alpha}) - g_2(\h{\alpha}) < y < \varphi(\h{\alpha}) - g_1 (\h{\alpha}) \}. 
\ees
Notice that for some values of $g_1( \h{\alpha})$, $g_2( \h{\alpha})$, the  corresponding set of $y$ can be empty. Also, let $p(x,y)=x$ be the projection to the base of the special flow. 
\begin{rem}\label{inclusionsDgs}
If $g_1'\leq g_1$ and $g_2'\geq g_2$, then $D_{\Phi}(g_1,g_2)\subset D_{\Phi}(g_1',g_2')$.
\end{rem}
% Recall (\S \ref{suspensionflowdef}) that $p$ denotes the projection to the base of the special flow.  

We will show that the limiting distribution (\ref{limiting}) exists and is given by 
\be\label{limitingdistribution}
P_N \left( (a,b) \times \{ c_0\} \times \dots \times \{ c_{N-1}\} \right) = \mu_3 \left(  D_{\Phi}(\ln a,\ln b)  \cap p^{-1} C_N \right).
\ee
%where we denote by 
%\bes
%D_{\Phi}(a,b) = \{ (\h{\alpha}, y) \in D_{\Phi} \st \varphi(\h{\alpha}) -\ln b < y < \varphi(\h{\alpha}) -\ln a  \}. 
%\ees
%(Remark that for certain values of $a, b, \h{\alpha}$ the set of corresponding $y$ can be empty).
Take  $\epsilon>0$. %Let us first localize $\h{\alpha}$. Remark that 
For each $n\in \mathbb{N}$, the cylinders $\{ {\Ci} \st \Ci\in \Cn \}$, constitute a countable partition of $D(\h{\G})$.
%To localize $\h{\alpha}$, 
Choose  $n$  so large that, by  Lemma \ref{variationflemma}, we have $| f (\h{\alpha}_1) -   f (\h{\alpha}_2) |\leq \epsilon/2$ for all $\h{\alpha}_1, \h{\alpha}_2 \in \Ci$. Let 
\bes
A_{\Ci } := \left\{ \h{\alpha} \in {\Ci} \st   a < \frac{q_{n_R(\h{\alpha})}(\h{\alpha} )}{R} < b, \quad \h{\G}^{n_R(\h{\alpha})-1}\h{\alpha} \in C_N \right\}.
\ees
By the second part of Lemma \ref{badpoints}, there exist $R_1>0$ such that if $R\geq R_1$ we have  $n_R(\alpha) > N $ for any $\alpha$ outside a set of $\mu_1$-measure less than $\epsilon$. Hence, by (\ref{conditionona}) (and Remarks \ref{mu2pimu1} and \ref{n_Rfuture}), we have
 %For each  $\hCi \in \\hCin $, let $T_{{\hCi}} = \ln R + f_{{\hCi}} $. 
%Remark that the cylinders $\{ {\Ci} \st \Ci\in \Cn \}$, give a partition of $D(\h{\G})$.
%For each  $\hCi \in \Cn $, let $T_{{\hCi}} = \ln R + f_{{\hCi}} $.
%Remark also that, given $\h{\alpha} = \{a_k\}_{k\in \mathbb{Z}}$, both $q_n_R(\alpha)$ and $n_R(\alpha)$ depend on $a_k$ with $k\geq 1$ only, hence they are constant on $\pi^{-1}\alpha$ 
%By Remarks \ref{mu2pimu1} and \ref{n_Rfuture}, as soon as $n_R(\alpha)\geq N$,
%Using also Lemma \ref{comparing}, we get, as long as $n_R\geq N$,
%\bes
%\begin{eqnarray}&&\nonumber
\bes
%\begin{split}
 \left| \mu_1  \left( \alpha  \st a < \frac{q_{n_R(\alpha)}}{R} < b,  \quad a_{n_R(\alpha)-k} = c_k, \, 1\leq k \leq N  \right) -% \nonumber \\&&\label{partition}\nonumber = 
%\\&=
\sum_{\Ci\in \Cn} \mu_2( A_{\Ci }) \right| \leq 2\epsilon .% \qquad \mathrm{where}
%\mu_2 \left( \h{\alpha} \in {\Ci} \backslash U \st a < \frac{q_{n_R}(\h{\alpha} )}{R} < b, \, \h{\G}^{n_R-1}\h{\alpha} \in C_N \right) .%
%\end{eqnarray}
%\end{split}
\ees
%where
%\bes
%\\ &&\nonumber
%A_{\Ci } := \{ \h{\alpha} \in {\Ci} \st   a < \frac{q_{n_R(\alpha)}(\h{\alpha} )}{R} < b, \quad \h{\G}^{n_R(\alpha)}\h{\alpha} \in C_N \}.
%\ees
%\end{eqnarray}
%where the last inequality follows by  Lemma \ref{comparing}.
Let us first reduce to a finite sum. Consider the finite subset  $\Cn^m$ of cylinders $\Ci=\C{a_{-n},\dots ; \dots, a_n}$ such that $a_i< m$ for $-n\leq i\leq n$. Since, if $\Ci \in \Cn\backslash \Cn^m$, there exists $-n\leq i\leq n$ and $k\geq m$ such that $a_i =k$, we have,  using Remark \ref{measureCn} and invariance of $\mu_2$,
\be \label{remindersum}
\sum_{\Ci\in \Cn \backslash \Cn^m} \mu_2({\Ci}  ) \leq   \sum_{i=-n}^{n} \sum_{k=m}^{\infty} \mu_2 \left( \h{\G}^{i} \hC{k} \right) \leq (2n+1)  \sum_{k=m}^{\infty} O\left(\frac{1}{k^2}\right) = O\left(\frac{1}{m}\right). 
\ee
Hence, choosing $m$ large enough, we can make (\ref{remindersum}) less than $\epsilon$. To each $\Ci \in \Cn^m$ we can apply Lemma \ref{comparing} and hence, for $R\geq \max_{\Ci \in \Cn^m} R_0 (\Ci)$ (where $R_0(\Ci)$ and $U_{\Ci}$ are as in Lemma \ref{comparing}) we get 
\begin{eqnarray}&& \label{partition}
\left|  \mu_1 \left( \alpha  \st a < \frac{q_{n_R(\alpha)}}{R} < b,  \quad a_{n_R(\alpha)-k} = c_k, \, 1\leq k \leq N  \right) - \sum_{\Ci\in \Cn^m } \mu_2 (A_{{\Ci} \backslash U_{\Ci} })\right| \leq
\nonumber \\&&\nonumber
\leq 2\epsilon + \left| \sum_{\Ci\in \Cn \backslash \Cn^m} \mu_2 (A_{\Ci} ) + \sum_{\Ci\in  \Cn^m} \mu_2 (A_{\Ci \cap U_\Ci }) \right| \leq 3 \epsilon + 7\epsilon 
%\right. \nonumber \\&&\label{partition} \nonumber  \left. -\sum_{\Ci\in \Cn \backslash \Cn^m } \mu_2 \left( \h{\alpha} \in {\Ci} \backslash U_{\Ci} \st a < \frac{q_{n_R(\alpha)}}{R} < b, \, \h{\G}^{n_R-1}\h{\alpha} \in C_N \right)  \right|
 \sum_{\Ci\in \Cn^m} \mu_2 ({\Ci} ) \leq  10 \epsilon ,
\end{eqnarray}
%\end{split}
%\ees
where the inequality  before the last follows from the observation that $A_\Ci\subset \Ci$, (\ref{remindersum})  and Lemma \ref{comparing}.

To conclude the proof and to get (\ref{limitingdistribution}), it is enough to prove that, for each $\Ci\in \Cn^m$, as long as $R$ is sufficiently large, we have
%of the terms in the finite sum over $\Cn^m$, % in (\ref{partition}) ,
% the following conditional probabilities satisfies:
\be \label{conditional}
\left|  \frac{\mu_2 (A_{\Ci \backslash U_{\Ci} } ) }{\mu_2 ({\Ci \backslash U_{\Ci} } )}
%\mu_2 \left(  a < \frac{q_{n_R}}{R} < b, \, \h{\G}^{n_R-1}\h{\alpha} \in C_N \, |\,  {\Ci} \backslash U_{\Ci} \right)
 -   \mu_3 \left(   D_{\Phi}(\ln a,\ln b) \cap p^{-1} C_N \right) \right| \leq \epsilon .
\ee
%is $(\epsilon\mu_2({\Ci}\backslash U_{\Ci}))$-close to $  \mu_3 \left(  \varphi^{-1} (\ln a,\ln b)  \cap p^{-1}C_N \right) \mu_2({\Ci}\backslash U_{\Ci})$ as $R\rightarrow \infty$.
%if $R$ is sufficiently large. 
Fix $\Ci \in \Cn^m $ and consider on $\Ci$ the function  $T= T(R) = \ln R - f_{{\Ci}} $ (recall that $f_{{\Ci}}= \sup_{{\Ci}} f $) and  $U=U_{\Ci}$ given by Lemma \ref{comparing}.  
Since  by Lemma \ref{comparing},  for $R\geq R_0(\Ci)$, on $\Ci\backslash U$ we have $n_R(\h{\alpha})=r(\h{\alpha}, T)$, applying (\ref{errorf}) of Lemma \ref{logqnandBS}, we get 
% and denoting by $\epsilon_{n,C}(\h{\alpha})= \epsilon_{n_R(\alpha)}(\h{\alpha}) + f_\Ci - f(\h{\alpha}) $.
\bes \begin{split} &
 \left\{ \h{\alpha} \in \Ci \backslash U \st a < \frac{q_{n_R(\h{\alpha})}}{R} < b \right\} =  \{ \h{\alpha} \in {\Ci} \backslash U \st \ln a < \ln q_{r(\h{\alpha}, T)}(\h{\alpha})-\ln R < \ln b  \} = \\ & = \{ \h{\alpha} \in {\Ci} \backslash U  \st \ln a  <  \BS{\varphi}{r(\h{\alpha}, T)} (\h{\alpha}) -T + \epsilon_{R,\Ci}(\h{\alpha}) < \ln b \},
     \end{split}
\ees
where we denoted by $\epsilon_{R,\Ci}(\h{\alpha})= \epsilon_{n_R(\h{\alpha})}(\h{\alpha}) - f_{\Ci} + f(\h{\alpha}) $. Let us show that $| \epsilon_{R,\Ci}|\leq 2 \epsilon$  uniformly on  $ \Ci \backslash U $; indeed, by construction of $\Ci$, $ | f(\h{\alpha}) - f_{\Ci} |\leq \epsilon$;  moreover, since by (\ref{r(T)inequality}) and Remark \ref{infsup}, $\ln R - f_{\Ci}  <  \BS{\varphi}{r(\h{\alpha}, T)} (\h{\alpha}) \leq M(\Ci)r(\h{\alpha}, T)  $, on  $\Ci\backslash U$ we have $n_R(\h{\alpha})=r(\h{\alpha}, T)\geq (\ln R - f_{\Ci})/M(\Ci)$, so enlarging $R_0(\Ci)$ we can assure that for each $\h{\alpha}\in \Ci\backslash U$ and $R\geq R_0$,  $n_R(\h{\alpha})$ is so large that, by Lemma \ref{logqnandBS}, also  $ \epsilon_{n_R(\h{\alpha})}(\h{\alpha})\leq \epsilon$.
% if $R\geq R_0$ by Lemmas \ref{variationflemma} and \ref{badpoints} and by construction of $\Ci$.

Let us denote by $(\Phi_t(x,y))^v$ the vertical component $y'$ of $\Phi_t(x,y)=(x',y')$. 
Using the definition of the flow action  (\ref{flowdef}) and the equality $n_R(\h{\alpha})=r(\h{\alpha}, T)$, % if we denote by $(\Phi(x,y))^v$ the vertical component $y'$ of $\varphi(x,y)=(x',y')$,
 we can rewrite 
%$ \BS{\varphi}{n_R} (\h{\alpha}) -T = \varphi(\h{\G}^{n_R-1}\h{\alpha}) - \Phi_{T_\Ci} (\h{\alpha},0) $.
\be \label{distance}
 \BS{\varphi}{r(\h{\alpha}, T)} (\h{\alpha}) -T = \varphi(\h{\G}^{n_R(\alpha)-1}\h{\alpha}) -\left( \Phi_{T} (\h{\alpha},0)\right)^v .
\ee 
Remark that $\h{\G}^{n_R(\h{\alpha})-1}\h{\alpha} = p (\Phi_{T} (\h{\alpha},0) )$ and that %,  if $\h{\alpha}\in A_{\Ci\backslash U} $,
condition  (\ref{conditionona}) can be expressed as $p (\Phi_{T} (\h{\alpha},0) )\in C_N$. % and that if $\h{\alpha}\in A_{\Ci\backslash U} $, then  $p (\Phi_{T} (\h{\alpha},0) )\in C_N$.
%, where $p(x,y)=x$ is the projection to the base of the special flow. 
The quantity (\ref{distance}) represents geometrically the vertical distance of $\Phi_{T} (\h{\alpha},0)$ from the roof function.%, since $\h{\G}^{n_R-1}\h{\alpha} = \pi \Phi_{T_\Ci} (\h{\alpha},0) $ .
%Let us denote by 
%\bes
%D_{\Phi}(a,b) = \{ (\h{\alpha}, z) \in D_{\Phi} \st \varphi(\h{\alpha}) -\ln b < z < \varphi(\h{\alpha}) -\ln a  \}. 
%\ees

Hence, recalling the definitions of the sets $D_{\Phi}(g_1,g_2)$ and $C_N$ given at the beginning of the proof, % and remarking that  if $\h{\alpha}\in A_{\Ci\backslash U} $,  (\ref{conditionona}) can be expressed as $p (\Phi_{T} (\h{\alpha},0) )\in C_N$,
  we have shown that %\BS{\varphi}{n_R} (\h{\alpha})$
\bes %\begin{split} 
%A_{\Ci\backslash U } &:= \{ \h{\alpha} \in {\Ci} \backslash U \st   a < \frac{q_{n_R}}{R} < b, \, \h{\G}^{n_R}\h{\alpha} \in C_N \}= \\ &
A_{\Ci\backslash U } :=
 %\\ & 
%\left( 
{\Ci} \backslash U \times \{0\}
%\right)
\, \cap \, \Phi_{-T} \left( D_{\Phi}(g_1, g_2) \cap p^{-1}{C}_N \right), \qquad g_1 = \ln a -\epsilon_{R,\Ci}, \quad g_2= \ln b -\epsilon_{R,\Ci}.
   %  \end{split}
\ees
%with $g_1 = \ln a -\epsilon_{R,\Ci}$ and $g_2= \ln b -\epsilon_{R,\Ci}$.
Let us ``thicken'' the sets in order to apply mixing of $\{\Phi_t\}_{t\in \mathbb{R}}$ as in the proof of Lemma \ref{comparing}.
 If we choose $0< \delta < \min \{ \min_{\h{\alpha}\in \Ci} \varphi(\h{\alpha}), \epsilon\} $ (well defined by Remark \ref{infsup}), for each $ \h{\alpha} \in A_{\Ci\backslash U}$ and $0\leq z < \delta$, reasoning as in the proof of Lemma \ref{comparing} (see e.g. (\ref{cases})), we get  $(\h{\alpha} ,z ) \in  \Phi_{-T} \left(  D_{\Phi}(g_1-\delta , g_2) \cap p^{-1}C_N  \cup  D_{\Phi}^{\delta}\right) $. Hence, using also Remark \ref{inclusionsDgs} combined with  $\delta\leq \epsilon $ and  $|\epsilon_{R,\Ci} |\leq 2\epsilon$,
\bes %\label{upper}
%\begin{eqnarray}
% \begin{split} 
\delta \mu_2(A_{\Ci\backslash U}) %=\mu_3(A_\Ci\times [0,\delta))
% &
 \leq  \mu_3  \left( {\Ci} \backslash U \times [0,\delta) \cap  \Phi_{-T} \left(  D_{\Phi}(\ln a + 3\epsilon , \ln b - 2 \epsilon) \cap p^{-1}C_N  \cup  D_{\Phi}^{\delta}\right) \right).%\leq % \xrightarrow{T \rightarrow \infty}
\ees
%\\ &
Remark that $\mu_3({D_{\Phi}^\delta})\leq \epsilon$. Enlarging $R_0$, so that  if $R\geq R_0$ also $T(R)$ is sufficiently large, we can use  mixing (see Proposition \ref{mixing}) to get%, remarking also that $\mu_3({D_{\Phi}^\delta})\leq \delta$.
\be \label{upper}
\delta \mu_2(A_{\Ci\backslash U}) \leq \delta \mu_2 ( {\Ci} \backslash U ) \left( \mu_3 \left(  D_{\Phi}(\ln a + 3\epsilon , \ln b - 2 \epsilon) \cap p^{-1}C_N \right) + 2\epsilon \right) .\ee
%\cup  D_{\Phi}^{\delta}
%    \end{split}
%where the last inequality holds if $R$ is large enough, so that also $T(R)$ is sufficiently to apply mixing (see Proposition \ref{mixing}) and uses also that $\mu_3({D_{\Phi}^\delta})\leq \delta$.
In order to get the opposite inequality, one can show,  reasoning again as in the proof of Lemma \ref{comparing},  that if $(\h{\alpha} ,z ) \in  {\Ci} \backslash U \times [0,\delta)$ is such that $(\h{\alpha} ,z ) \in \Phi_{-T} \left(  D_{\Phi}(g_1 , g_2-\delta ) \cap p^{-1}C_N  \backslash  D_{\Phi}^{\delta}\right) $, we have that $\h{\alpha} \in A_{\Ci\backslash U}$. This means, using again also Remark \ref{inclusionsDgs}, that 
\bes
  {\Ci} \backslash U \times [0,\delta) \cap \Phi_{-T} \left(  D_{\Phi}(\ln a -3\epsilon , \ln b + 2 \epsilon) \cap p^{-1}C_N  \backslash  D_{\Phi}^{\delta}\right) \subset A_{\Ci\backslash U} \times [0,\delta) .
\ees
Applying again mixing, enlarging again $R_0$ if necessary, for $R\geq R_0$, (and using that for any measurable $D\subset D_{\Phi}$, we have $\mu_3\left( D\backslash  D_{\Phi}^{\delta}\right) \geq \mu_3(D) - \epsilon$) we get
\be \label{lower}
\delta \mu_2(A_{\Ci\backslash U}) \geq \delta  \mu_2 ( {\Ci} \backslash U ) \left( \mu_3  \left(  D_{\Phi}(\ln a -3\epsilon , \ln b + 2 \epsilon) \cap p^{-1}C_N \right)
% \backslash  D_{\Phi}^{\delta} 
- 2 \epsilon 
\right).
\ee
Since moreover, by Fubini theorem,
%Remarking that $\mu_3( D_{\Phi}^{\delta} ) \leq \delta$ and  
%\bes \left| \mu_3 \left(  D_{\Phi}(\ln a +2\epsilon , \ln b - 3 \epsilon) \cap p^{-1}C_N \right) - \mu_3 \left(  D_{\Phi}(\ln a  , \ln b ) \cap p^{-1}C_N \right) \right| \leq 5 \epsilon .
%\ees
\bes \left| \mu_3 \left(  D_{\Phi}(\ln a \pm 3 \epsilon , \ln b \mp 2 \epsilon) \cap p^{-1}C_N \right) - \mu_3 \left(  D_{\Phi}(\ln a  , \ln b ) \cap p^{-1}C_N \right) \right| \leq 5 \epsilon ,
\ees
combining  (\ref{upper}) and (\ref{lower}) we get (\ref{conditional}) and hence conclude the proof of the existence of the limiting distribution.
\end{proofof}
%\ep \medbreak

\section{Mixing of the special flow.} \label{mixingsection}
In what follows we briefly outline the proof of Proposition \ref{mixing} given in \cite{DS:sta}. %We recall the proof in \S \ref{mixingsection}.
%\proc{Proof of Proposition \ref{mixing}.}
\begin{proofof}{Proposition}{mixing}
Given a point $(\h{\alpha}_0, y_0)\in D_{\Phi}$, let us construct the local stable and unstable leaves through it, denoted by $\Gamma^{(s)}_{loc}(\h{\alpha}_0, y_0)$ and $\Gamma^{(u)}_{loc}(\h{\alpha}_0, y_0)$ respectively (as a general reference, see e.g. \cite{Si:top}).

Since the roof function $\varphi(\h{\alpha})$ depends only on $(\h{\G}\h{\alpha})^-$, it is easy to construct the local unstable leaf, which is given by a piece of a ``horizontal'' segment:
\be\label{locstab}
\Gamma^{(u)}_{loc}(\h{\alpha}_0, y_0) \subset \{ (\h{\alpha}, y) \st \h{\alpha}^-=\h{\alpha}_0^-, \,  y=y_0 \}.
\ee
The local stable leaf through $(\h{\alpha}_0, y_0)$ is given locally by the following curve parametrized by $\alpha^-$:
\be \label{locunst}
\Gamma^{(s)}_{loc}(\h{\alpha}_0, y_0) \subset \left\{ (\h{\alpha}, y) \st \h{\alpha}^+=\h{\alpha}_0^+, \,  y= y_0 + \ln \frac{1+ \h{\alpha}^- \h{\alpha}^+_0}{1+ \h{\alpha}_0^- \h{\alpha}^+_0} \right\}.
\ee
In order to see it, one can construct it as follows. 
Let us denote by $(\h{\alpha}_t, y_t) = \Phi_t (\h{\alpha}_0, y_0)$. Consider a small ``vertical'' segment at  $(\h{\alpha}_t, y_t)$, i.e.
\bes
\Gamma^t_{\delta_t}  = \{ (\h{\alpha}, y_t) \st \h{\alpha}^+=\h{\alpha}_t^+, \,  |\h{\alpha}^- - \h{\alpha}_t^-|< \delta_t \},
\ees
where $\delta_t$ is chosen sufficiently small so that, for some $\delta>0$,
%$\Phi^{-t} (\Gamma^t_{\delta_t}  )$ consists of points for which 
\bes
\Phi_{-t} (\Gamma^t_{\delta_t}  ) \subset \{ (\h{\alpha}, y) \st \h{\alpha}^+=\h{\alpha}_0^+,\,  | y-y_0| < \delta,  \, 0< y_0 -\delta < y < \varphi(\h{\alpha_0}) -\delta  \}.
\ees
Then, if $(\h{\alpha} , y)\in \Phi^{-t} (\Gamma^t_{\delta_t}  )$, by definition of special flow, since $r(t)(\h{\alpha}) = r(t)(\h{\alpha_0}) = r(t)$ by construction, we have $y - \BS{\varphi}{r(t)}(\h{\alpha}) = t = y_0 - \BS{\varphi}{r(t)}(\h{\alpha}_0)$.
%\bes
%y - \BS{\varphi}{r(t)}(\h{\alpha}) = t = y_0 - \BS{\varphi}{r(t)}(\h{\alpha}_0).
%\ees
Denote $\alpha_0^+= \{a^0_k\}_{k\in \mathbb{N}_+}$ and ${p_n}/{q_n}$ %$\left\{\frac{p_n}{q_n}\right\}_{n\in\mathbb{N}_+}$
 its convergents. Let 
%$\beta'=[a^0_1 + \h{\alpha}^-_0, a^0_2, \dots, a^0_{r(t)} ]$ and $\beta''=[a^0_1 + \h{\alpha}^-, a^0_2, \dots, a^0_{r(t)} ]$ 
\be\label{betas}
\beta':=[a^0_1 + \h{\alpha}^-_0, a^0_2, \dots, a^0_{r(t)} ]=  \frac{1}{\h{\alpha}^-_0 + \frac{q_{r(t)}}{p_{r(t)}}}, \, \, \beta'':=[a^0_1 + \h{\alpha}^-, a^0_2, \dots, a^0_{r(t)} ]= \frac{1}{\h{\alpha}^- + \frac{q_{r(t)}}{p_{r(t)}}}
\ee
and  
%$\left\{\frac{p'_n}{q'_n}\right\}_{1\leq n\leq r(t)}$ and  $\left\{\frac{p''_n}{q''_n}\right\}_{1\leq n\leq r(t)}$ 
${p'_n}/{q'_n}$ and ${p''_n}/{q''_n}$ %, %  $1\leq n\leq r(t)$
their respective convergents.  Remark that $p'_{n }= p''_{n} = p_{n} $ for  $1\leq n\leq r(t)$ since they satisfy the same recursive equations  $p_{k+1} = a_{k+1} p_k + p_{k-1}$  for $2\leq k+1\leq r(t)$ with  initial data $p_0=0$, $p_1=1$. Hence, $\beta'= \frac{p_{r(t)}}{q'_{r(t)}}$ and  $\beta''= \frac{p_{r(t)}}{q''_{r(t)}}$.
%Moreover, we have $\beta'= \frac{p_{n(t)}}{q'_{n(t)}} = \frac{1}{\h{\alpha}}^-_0 + \frac{q_{r(t)}}{p_{r(t)}}$ and $\beta''= \frac{p_{n(t)}}{q''_{n(t)}} = \frac{1}{\h{\alpha}}^- + \frac{q_{r(t)}}{p_{r(t)}}$. 
Using (\ref{logqn}, \ref{BS})  and (\ref{betas}), one gets
%Comparing with (\ref{logqn, logBS}), one gets
\bes
y= y_0 +  \BS{\varphi}{r(t)}(\h{\alpha}) -  \BS{\varphi}{r(t)}(\h{\alpha}_0) = y_0 + \ln \frac{q''_{r(t)}}{q'_{r(t)}} = y_0 + \ln \frac{1+ \h{\alpha}^- \frac{p_{r(t)}}{q_{r(t)}}}{1+ \h{\alpha}_0^-  \frac{p_{r(t)}}{q_{r(t)}}} .
\ees
 As $t$, and hence $r(t)$, tend to infinity, $ \frac{p_{r(t)}}{q_{r(t)}}$ converge to $\alpha^+_0$ and we get (\ref{locunst}).
%Remark that $p'_{n }= p''_{n} = p_{n} $ for all $n\in \mathbb{N_+}$ since they satisfy the same recursive equations  $p_{k+1} = a_{k+1} p_k + p_{k-1}$  for $k+1\leq 2$ and initial data $p_0=0$, $p_1=1$. Hence, remarking that $\beta' = $
%The latter expression can be rewritten analogously to (\ref{BSdifference}) and then, using the identities on continued fractions (\cite{Kh:con}), as
%\bes
%\BS{\varphi}{r(t)}(\h{\alpha})- \BS{\varphi}{r(t)}(\h{\alpha}_0) = \ln \frac{1+ \alpha^- \alpha^+_{r(t)-1}}{1+ \alpha_0^- \alpha^+_{r(t)-1}},  
%\qquad \alpha^+_{r(t)-1} = [a^0_1, a^0_2, \dots, a^0_{r(t)}],
%\ees
%%with $\alpha^+_{r(t)-\alpha^+_{r(t)-1}1} = [a^0_1, a^0_2, \dots, a^0_{r(t)}]$,
% where $\alpha_0^+= \{a^0_k\}_{k\in \mathbb{N}_+}$. As $t$, and hence $r(t)$, tend to infinity, $\alpha^+_{r(t)-1}$ converge to $\alpha^+_0$ and we get (\ref{locunst}).%

The global unstable and stable leaves can be obtained as
\bes
\Gamma^{(u)}(\h{\alpha}_0, y_0)  = \bigcup_t \Phi_{t} \Gamma^{(u)}_{loc}(\h{\alpha}_{-t}, y_{-t}); \qquad \Gamma^{(s)}(\h{\alpha}_0, y_0)  = \bigcup_t \Phi_{-t} \Gamma^{(s)}_{loc}(\h{\alpha}_t, y_t).
\ees

To prove mixing, it is enough to show that the stable and unstable foliations form a non-integrable pair. From their non-integrability, it follows from the general theory  (see \cite{Si:top}) that the Pinsker partition is trivial and hence that $\{\Phi_t\}_{t\in \mathbb{R}}$ is a $K$-flow and, in particular,  is mixing.

Consider a sufficiently small neighborhood $\mathscr{U}(\h{\alpha_0}, y_0)\subset D_{\Phi}$ of $(\h{\alpha_0}, y_0)$. It is enough to show that, for a positive measure set of $(\h{\alpha}, y)\in \mathscr{U}(\h{\alpha_0}, y_0)$, $(\h{\alpha}, y)$ can be connected to $(\h{\alpha}_0, y_0)$ through a segments of local stable and unstable leaves, in particular if there exist  $(\h{\alpha}_i, y_i)\in \mathscr{U}(\h{\alpha_0}, y_0)$, $i=1,2$, such that $(\h{\alpha}_1, y_1) \in \Gamma^{(s)}(\h{\alpha_0}, y_0 )$, $(\h{\alpha}_2, y_2) \in \Gamma^{(u)}(\h{\alpha_1}, y_1 )$ and   $(\h{\alpha}, y) \in \Gamma^{(s)}(\h{\alpha_2}, y_2 )$.

Using explicitly the equations (\ref{locstab},\ref{locunst}), one can check that these points exist as soon as we can find $y_1$ and $\h{\alpha}_1^-$ such that $((\h{\alpha}_1^-, \h{\alpha}_0^+), y_1)\in \mathscr{U}(\h{\alpha_0}, y_0)$ and
\be\label{findalpha1}
y_1 = y_0 + \ln \frac{1+\h{\alpha}_1^- \h{\alpha}_0+}{1+\h{\alpha}_0^- \h{\alpha}_0+}, \qquad y = y_1 + \ln \frac{1+\h{\alpha}_1^- \h{\alpha}+}{1+\h{\alpha}^- \h{\alpha}+}, 
\ee
since in this case we can take $({\h{\alpha}}_1, y_1)  = ((\h{\alpha}_1^-, \h{\alpha}_0^+), y_1 )$ and $({\h{\alpha}}_2, y_2)  = ((\h{\alpha}_1^-, \h{\alpha}^+), y_1 )$. Equations (\ref{findalpha1}) can be solved if 
\be\label{hypsurf}
\frac{\h{\alpha}^+}{1+\h{\alpha}^- \h{\alpha}^+ } e^{y} \neq \frac{\h{\alpha}_0^+}{1+\h{\alpha}_0^- \h{\alpha}_0^+ } e^{y_0}.
\ee
The points $(\h{\alpha}, y)$ for which there is the equality in (\ref{hypsurf}) lie on a surface in $D_{\Phi}$ and hence have measure zero.
This concludes the proof of the non integrability. 
%\ep\medbreak
\end{proofof}

\section{Concluding Remark.}
Let $T$ be an ergodic automorphism of the measure space $(M,\mathscr{M}, \mu)$ and $f\in L^1(M,\mathscr{M}, \mu)$, $\int f\ud x >0$. The following problem is a generalization of a classical renewal problem in probability theory. Take $R>0$ and consider the first $n_R$ such that 
\bes f(x)+ f(Tx) + \dots + f(T^{n_R} x)>R.\ees
 What will be the limiting distribution of  $f(x)+ f(Tx) + \dots + f(T^{n_R} x)- R$ as $R$ tends to infinity?  
The answer can be given in terms of a special flow which is similar to the one considered above. Interesting aspects of this problem appear when $\int |f| \ud x = \infty$.

%where $T$ is the Gauss map and $f(x)= [\frac{1}{x}]$, so that the Birkhoff sum reduces to the sum of the first $n$ digits of the continued fraction expansion,
Results concerning the limiting distribution when considering the Gauss map and the sum of the entries of the continued fraction expansion can be found in \cite{KS:dis}.
% and large deviations for the Gauss map, 
%A similar problem where $T$ is the Gauss map and $f(x)= [\frac{1}{x}]$, so that the Birkhoff sum reduces to the sum of the first $n$ digits of the continued fraction expansion, has been considered in \cite{KS:dis}.

%The answer can be given in terms of a special flow which is similar to the one considered above. Interesting aspects of this problem appear when $\int |f| \ud x = \infty$.

%A similar problem where $T$ is the Gauss map and $f(x)= [\frac{1}{x}]$, so that the Birkhoff sum reduces to the sum of the first $n$ digits of the continued fraction expansion, has been considered in \cite{KS:dis}.

%\ack{
\subsection*{Acnkowledgments.}
The first author thanks NSF Grant DMS $0600996$ for the financial support.

\bibliography{bibliorenewal}
\bibliographystyle{amsplain} 

%{\it Email addresses:}\quad {sinai@math.princeton.edu, \quad ulcigrai@math.princeton.edu}
\end{document}